\documentclass{amsart}
\usepackage{euscript}
\usepackage[all]{xy}
\usepackage{graphicx}
\usepackage{tikz}
\usetikzlibrary{arrows,patterns}
\usepackage{xcolor}

\usepackage{color}
\usepackage{amsmath}
\usepackage{amssymb}
\usepackage{amsfonts}
\usepackage{mathrsfs}
\usepackage{amsthm}
\usepackage[english]{babel}
\usepackage[latin1]{inputenc}
\usepackage{verbatim}

\newtheorem{theorem}{Theorem}[section]
\newtheorem{definition}[theorem]{Definition}

\newtheorem{lemma}[theorem]{Lemma}
\newtheorem{coro}[theorem]{Corollary}

\newcommand{\Z}{{\mathbb{Z}}}

\newcommand{\N}{{\mathbb{N}}}

\begin{document}

\title{The Minimal Euclidean Function on the Gaussian Integers}
\author{Hester Graves}
\date{\today}

\maketitle

\section{Introduction}

The Euclidean algorithm, an ancient piece of mathematics, generalizes to the 
modern concepts of Euclidean functions and Euclidean domains.  We say a domain, $R$, 
is Euclidean if there exists a (Euclidean) function $f$,
$f: R \setminus 0 \rightarrow \N,$ such that if $a \in R$ and $b \in R \setminus 0$, 
then there exist some $q,r \in R$ such that $a =qb +r$, with either $r=0$ or $f(r) < f(b)$.  We restate the standard definition as follows.

\begin{definition} (\cite{Motzkin})
Given a domain $R$, a function $\phi: R\setminus 0 \rightarrow \N$ is \textbf{Euclidean} if, for any $b \in R\setminus 0$ and any $[a] \in R/b$, then either $b | a$ or there exists some $r \in R\setminus 0$ such that $[a]=[r]$ and $\phi(r) < \phi(b)$.
\end{definition}

In other words, if $b \in R \setminus 0$, then every non-zero coset $[a] \in R/b$ has a representative $r$ such that $\phi(r) < \phi(b)$.  This reformulation of the definition paves the way for Motzkin's Lemma and his construction, defined below.

\begin{definition} \textbf{Motzkin's Construction}  Given a domain $R$, we define 
\begin{align*}
A_{R,0} &: = 0 \cup R^{\times} \\
A_{R,j} &: = A_{R, j-1} \cup \{ \beta :A_{R,j-1} \twoheadrightarrow R/\beta \}, \text{ and}\\ 
A_R & := \bigcup_{j=0}^{\infty} A_{R,j}.
\end{align*}
\end{definition}

We remind ourselves that $A_{R,j} \rightarrow R/ \beta$ if, for every $[a] \in R/\beta$, there exists some $r \in A_{R, j}$ such that $[a] = [r]$. Our set $A_{R, j}$ is the union of $A_{R, j-1}$ and all $\beta$ such that every coset in $R/\beta$ has a representative in $A_{R, j-1}$.  This construction allows us to state Motzkin's Lemma, first published in 1949 (\cite{Motzkin}).

\begin{lemma} (\cite{Motzkin}) \label{Motzkins_Lemma}  A domain $R$ is Euclidean if and only if $R = A_R$.  Furthermore, if $F$ is the set of all Euclidean functions on $R$, and if 
\begin{align*}
\phi_R &: R \setminus 0 \rightarrow \N,\\ 
\phi_R(a) &:= j \text{ if }a \in A_{R,j} \setminus A_{R, j-1},
\end{align*}
then $\phi_R(a) = min_{f\in F} f(a)$ and $\phi_R$ is itself a Euclidean function.
\end{lemma}

Motzkin's Lemma and construction changed mathematicians' approaches toward Euclidean domains.  Previously, proving a domain was Euclidean was an exercise in trial and error, as people searched for potential Euclidean functions.  Motzkin let us know that if such a function existed, then there is a definition we can supposedly use to compute the minimal such function, $\phi_R$.  In his paper ``The Euclidean Algorithm,'' he went on to show that $\phi_{\mathbb{Z}}(x) = 
\lfloor \log_2|x| \rfloor$, or one less than the number of digits in the binary expansion of $x$.  

Since Gauss showed that $\mathbb{Z}[i]$ is Euclidean, it seems like the natural next step of Motzkin's work would be to show that $\phi_{\mathbb{Z}[i]}$ is one less than the number of digits in the $(1+i)$-ary expansion, as $2 = i(1+i)^2$.  HW Lenstra found an elegant, abstract description of the sets $A_{\mathbb{Z}[i], n}$ in his 1974 work ``Lectures on Euclidean Rings,'' using advanced analytic and algebraic techniques.  M. Fuchs wrote sophisticated code in his thesis to inductively compute $\phi_{\mathbb{Z}[i]}$.  

This paper introduces an alternate, concrete description of the sets $A_{\mathbb{Z}[i],n}$ that allows us to quickly and easily compute the function $\phi_{\mathbb{Z}[i]}$.  We then introduce a new proof using purely elementary techniques. 


\section{Properties of $\mathbb{Z}[i]$}

In many ways, $\mathbb{Z}$ is an easy ring to study.  The ring has only two units, $\pm 1$, and we can distinguish between the two.  They make adding in $\mathbb{Z}$ very easy -- if $u$ and $v$ are elements of $\{ 0, \pm 1\}$, then $u2^n + v2^n$ is an element of $\{0, \pm 2^{n+1}\}$.  This gives us the luxury of a unique binary expansion for every positive integer.  As the minimal Euclidean function on the integers is the highest power of two in said expansion, we can easily see that 
\begin{align*}
A_{\mathbb{Z}, n } &= \{ x \in \mathbb{Z} : \lfloor \log_2 |x| \rfloor < n\}\\
& = \{ x \in \mathbb{Z} : |x| \leq 2^{n+1} -1\}.
\end{align*}

In contrast, the Gaussians have four multiplicative units, the set $\{ \pm 1, \pm i \}$, and it is impossible to distinguish between $i$ and $-i$.  Complex conjugation, the map that sends $i$ to $-i$, is a consequence of this confusion; we denote the complex conjugate of $a+bi$ by $\overline{a+bi} = a -bi$.
This makes adding in $\mathbb{Z}[i]$ tricky--if $u$ and $v$ are in $\{\pm 1, \pm i\}$,$u \neq v$, then 
$u(1+i)^n + v(1+i)^n \in \{w(1+i)^{n+1}, w(1+i)^{n+2}: w \in \Z[i]^{\times} \}$.  We cannot have a unique $(1+i)$-ary expansion when the choice of $(1+i)$ versus $-i(1+i) = 1-i$ is arbitrary.  As an illustration, 
\begin{align*}
4+i &= -(1+i)^4 +i \\
&= i(1+i)^2 + (1+i) +1,
\end{align*}
and there is no obvious method to choose the expansion with the lowest degree of $(1+i)$.  Our first object of study, therefore, are the following sets.

\begin{definition} We define the sets $B_n$ to be the Gaussian integers that can be written with $n+1$ `digits,' i.e. 
$$B_n = \left \{ \sum_{j=0}^n v_j (1+i)^n, v_j \in \{0, \pm 1, \pm i\} \right \}.$$
\end{definition}

The sets $B_n$ have several nice properties.  They are closed under complex conjugation and multiplication by elements of $B_0 = \{0, \pm 1, \pm i\}$.  If $a+bi \in B_n$, then $(1+i)^j (a+bi) \in B_{n+j}$.  Similarly, if $2^j$ divides both $a$ and $b$ for some $a+bi \in B_n$, then 
$\frac{a}{2^j} + \frac{b}{2^j} i \in B_{n-2j}$.  These lead to our first result on the sets $B_n$. 



\begin{lemma}  If $x \in B_{k}$ and $(1+i)^{k+1} |x$, then $x =0$.
\end{lemma}
\begin{proof} Since $B_0 = \{0, \pm 1, \pm i\}$, the claim is clearly true if $k=0$.  We will now prove the claim by induction.  
Let $k \geq 1$ and let us assume that the claim is true for all $j$, $0 \leq j<k$, and 
let $(1+i)^{k+1}q \in B_k$.  By definition, $$(1+i)^{k+1}q = \sum_{j=0}^{k} (1+i)^j w_j,$$ where $w_j \in B_0$.  
Since $(1+i)$ divides both $(1+i)^{k+1}q$ and $\sum_{j=1}^{k} (1+i)^j w_j$, $(1+i)$ must divide $w_0$, and therefore $w_0$ is equal to zero.
Therefore, $$(1+i)^{k+1}q = \sum_{j=0}^{k-1} (1+i)^j w_{j+1},$$
but since the left hand side is divisible by $(1+i)^k$ and the right hand side is an element of $B_{k-1}$, both sides must equal zero.
We conclude that $(1+i)^{k+1}q$ is zero. 
\end{proof} 

Unfortunately, our definition of $B_n$ does not allow us to easily determine whether a given element $a+bi$ is, indeed, a member of $B_n$.  In order to do so, we introduce the sequence below.

\begin{definition}  We define the sequence 
\begin{equation*}
w_k = 
\begin{cases}
2^{n+1} + 2^n & \text{if}\ k=2n\\
2^{n+2} & \text{if}\ k=2n+1.
\end{cases}
\end{equation*}
\end{definition}
Note that if $k \geq 2$, then $w_k = 2 w_{k-1}$.  The sequence has some other useful properties, as well: 
\begin{align*}
2(w_{n+1} - w_n) &\leq w_n, \\
w_n + 2^{\lfloor n/2 \rfloor} &\leq w_{n+1}, \\
w_n - 2^{\lfloor n/2 \rfloor +1} &= w_{n-1}, \text{ and}\\
 w_{n+1} - 3 \cdot 2^{\lfloor n/2 \rfloor } &\leq w_{n-1}.
\end{align*}We can use this sequence to define the following geometric object, which will be helpful in section 5, proving the main result.

\begin{definition}  We define the $n^{th}$ octogon to be 
\begin{equation*}
Oct_n: = \{ x+yi: |x| + |y| \leq w_n -2, |x| + |y| \leq w_{n+1} - 3 \}.
\end{equation*}
\end{definition}

The sequence $w_n$ gives us a criterion to determine whether a given $a+bi$ is an element of $B_n$.

\begin{theorem}\label{B_union}
The set $B_n \setminus 0$ equals the union
\begin{equation*}\label{big_union}
\displaystyle
\bigcup_{j=0}^{\lfloor n/2 \rfloor}
\{ x+iy: 2^j \parallel (x,y); |x|, |y| \leq w_n - 2^{j+1}; |x|+|y| \leq w_{n+1} - 3\cdot 2^j \}.
\end{equation*}
\end{theorem}

\begin{proof}
We will prove this by induction on $n$.  We will first show that this holds when $n$ is even, and then we will show it for when $n$ is odd.
Note that 
\begin{align*}
B_0 \setminus 0 &= \{ \pm 1, \pm i \} \\
&=\{ x+iy : 1 \parallel (x,y); |x|,|y| \leq w_0-2=1; |x|+|y| \leq w_1 -3 = 1 \} \\
\intertext{and}
B_1 \setminus 0 
&= \{ \pm 1, \pm i, \pm 1 \pm i, \pm 2 \pm i, \pm 1 \pm 2i \} \\
&= \{ x+iy : 1 \parallel (x,y); |x|,|y| \leq w_1-2=2; |x|+|y| \leq w_2-3 = 3\} .
\end{align*}

\underline{\textbf{Case} $\mathbf{n =2k}$}: First suppose that $k \geq 1$ and that the theorem holds for all $j$, $0 \leq j < n=2k$.  If $a+bi \in B_{2k} \setminus B_{2k-1}$, then there exists a unit $u \in \mathbb{Z}[i]^{\times}$ such that $(ua-2^k) + bi \in B_{2k-1}$.  We can therefore assume, without loss of generality, that $(a-2^k) + bi \in B_{2k-1}$.  Thus there exists some $0 \leq j < k$ such that 
\begin{align*}
2^j \parallel (a-2^k, b); |a-2^k|, |b| \leq w_{2k-1} - 2^{j+1};&\text{ and } |a-2^k| + |b| \leq w_{2k} - 3\cdot 2^j.\\
\intertext{This implies that}
2^j \parallel (a,b); |a|, |b| \leq w_{2k-1} + 2^k - 2^{j+1}; &\text{ and }
|a|+|b| \leq w_{2k} + 2^k - 3 \cdot 2^j,\\
\intertext{and thus}
2^j \parallel (a,b); |a|, |b| \leq w_{2k} - 2^{j+1};&\text{ and }
|a|+|b| \leq w_{2k+1} - 3 \cdot 2^j,
\end{align*}
demonstrating that $B_{2k} \setminus 0$ is contained in the union in the theorem statement.

Now let $0 \leq j <k$ and let $a+bi$ satisfy $2^j \parallel (a,b)$; $|a|+|b| \leq w_{2k}-2^{j+1}$; and $|a|+|b| \leq w_{2k+1} - 3\cdot 2^j$.  Clearly, we can assume that $a \geq b \geq 0$.  Note that 
$2 \nmid (a/2^j, b/2^j)$; $\frac{a}{2^j} + \frac{b}{2^j} \leq w_{2(k-j)} -2$; and 
$\frac{a}{2^j} + \frac{b}{2^j} \leq w_{2(k-j) +1} -3$.  If $j > 0$, then 
$\frac{a}{2^j} + \frac{b}{2^j} i \in B_{2(k-j)}$ by our induction hypothesis and therefore $a+bi \in B_{2k}$.

If $j =0$, then $a$ and $b$ satisfy
$2 \nmid (a,b)$; $a, b \leq w_{2k} -2$;  and $a+b \leq w_{2k+1}-3$, implying that $ 0 \leq b \leq 2^{k+1} -2$.
Furthermore, if our pair $a,b \leq w_{2k-1}-2$ and $a+b \leq w_{2k -3}$, then $a+bi \in B_{2k-1} \subset B_{2k}$, so we are now reduced the cases where either 
$a > w_{2k-1}-2$ or $a+b > w_{2k}-3$.

In both of these scenarios, $2^{k-1} \leq a -2^k \leq 2^{k+1} -2$, so 
$$|a-2^k|, |b| \leq 2^{k+1}-2 = w_{2k-1} -2, \text{ and}$$
$$|a-2^k| + |b| = a+b-2^k \leq w_{2k+1}-3-2^k = 2^{k+1} + 2^k - 3 = w_{2k} -3.$$
The pair $a$ and $b$ have different parities, so $2 \nmid (a-2^k, b)$ and thus $(a-2^k) + bi \in B_{2k-1}$.  We therefore conclude that 
$a+bi \in (B_{2k-1} + 2^k) \subset B_{2k}$, and that $B_{2k}$ is indeed equal to the union in the theorem statement.

We assumed that $k\geq 1$ and that the theorem held for all $j$, $0 \leq j < 2k$.  After the first part of the proof, we can now say (under the same assumption), that the theorem holds for all $j$, $0 \leq j \leq 2k$.

\underline{\textbf{Case} $\mathbf{n=2k+1}$}:  Let $a+ bi \in B_{2k+1}\setminus B_{2k}$; we can assume without loss of generality that 
$(a+bi) - 2^k(1+i) \in B_{2k}$.  Applying our induction hypothesis, we know that there exists some $j$, $0 \leq j \leq k$, such that 
$2^j \parallel (a-2^k, b-2^k)$; $|a-2^k|,|b-2^k| \leq w_{2k}-2^{j+1}$; and $|a-2^k| + |b-2^k| \leq w_{2k+1} - 3\cdot 2^j$.

Note that if $j=k$, then $|a-2^k| + |b-2^k| \leq 2^k$, so one of the summands must be zero.  This implies that for all our $j$'s,
\begin{align*}
2^j \parallel (a, b); |a|,|b| \leq w_{2k+1}-2^{j+1}; \text{ and } |a| + |b| \leq w_{2k+2} - 3\cdot 2^j.
\end{align*}
We infer that $B_n = B_{2k+1}$ is contained inside the appropriate union.

To prove the other direction of containment, let $0 \leq j \leq k$ and let $a+bi$ satisfy $2^j \parallel (a,b)$; $|a|,|b| \leq w_{2k+1} - 2^{j+1}$;
and $|a|+|b| \leq w_{2k+2} - 2\cdot 2^j$.  Because the union in equation (\ref{big_union}) is closed under complex conjugation and multiplication by units, we can assume that $a \geq b \geq 0$.  If $j \geq 1$, then $2 \nmid (\frac{a}{2^j}, \frac{b}{2^j})$; $|\frac{a}{2^j}|, |\frac{b}{2^j}| \leq w_{2(k-j) +1 } -2$; and $|\frac{a}{2^j}| + |\frac{b}{2^j}| \leq w_{2(k-j) +2} -3$.  Our induction hypothesis implies that 
$\frac{a}{2^j} + \frac{b}{2^j}i \in B_{2(k-j)+i}$, and thus $a+bi \in B_{2k+1}$.

If $j=0$, $a \leq w_{2k}-2$, and $a+b \leq w_{2k+1} -3$, then $a+bi \in B_{2k} \subset B_{2k+1}$ by our induction hypothesis. If either $a>w_k -2$ or $a+b > w_{2k+1} -3$,
 then $a \geq 2^{k+1} -2 $, $2^k - 2 \leq a-2^k \leq w_{2k} - 3$, and thus 
$|a-2^k|, |b-2^k| \leq w_{2k}-3$.  If $b \geq 2^k$, then 
\begin{equation*}
|a - 2^k| + |b-2^k| = a+b - 2^{k+1} \leq w_{2k+2} - 2^{k+1} - 3 = w_{2k+1} -3.
\end{equation*}
If $b < 2^k$, then 
\begin{equation*}
|a-2^k| + |b-2^k| \leq a-2^k + 2^k = a \leq w_{2k+1} - 2.
\end{equation*}  The pair $a$ and $b$ have different parities, so the sum of 
$|a-2^k| + |b-2^k|$ must be odd, and is therefore bounded above by $w_{2k+1} -3$.  Putting the previous statements together, if $j =0$ and $a+bi \notin B_{2k}$, then $2 \nmid (a-2^k, b-2^k)$; $|a-2^k|, |b-2^k| \leq w_{2k}-3$; and $|a-2^k| + |b-2^k| \leq w_{2k+1} -3$.  We conclude that $(a-2^k) + (b-2^k)i \in B_{2k}$, so $a+bi \in B_{2k+1} + 2^k(1+i) \subset B_{2k+1}$.  
\end{proof}

We can also view this result geometrically.  The set $B_n$ forms a lacy octogon; Martin Fuchs has a nice illustration on page 72 of \cite{Fuchs}.
\begin{coro}\label{oct} The set $B_n$ satisfies
\begin{equation*}
B_n \subset Oct_n \subset B_{n+1} \cup \{ (1+i)^{n+2} (\mathbb{Z}[i])^{\times} \}.
\end{equation*}
If $a + bi \in Oct_n$, $2 \nmid (a,b)$, then $a+bi \in B_n$.
\end{coro} 

\begin{proof}  The theorem clearly implies that $B_n \subset Oct_n$.  Suppose $a+bi \in Oct_n$, so that $|a|, |b| \leq w_n -2$ and $|a| + |b| \leq w_{n+1} -3$.  If $n =2k$ and $2^j \parallel (a,b)$ with $j \leq k$, we see that 
\begin{align*}
|a|,|b| &\leq 2^{k+1} + 2^k - 2^j \text{ and} \\
|a| + |b| &\leq 2^{k+2} - 2^j, \text{ so}\\
\end{align*}
\begin{align*}
|a|, |b| &\leq 2^{k+1} + 2^k - 2^j &= 2^{k+2} - 2^k - 2^j &\leq w_{n+1} - 2^{j+1} \text{ and}\\
|a|+|b| &\leq 2^{k+2} - 2^j        &= (2^{k+2} + 2^{k+1}) - 2^{k+1} - 2^j &\leq w_{n+2} - 3\cdot 2^j.
\end{align*}
If $j>k$, then $j = k+1$ and we have $a+bi \in 2^{k+1}(\mathbb{Z}[i]^{\times})$, and none of those elements are in $B_{n+1}$ as $\gcd (a,b) = 2^{k+1}$. 

If $n = 2k+1$ and $2^j \parallel (a,b)$ (this time w don't need $j \leq k$), then 
\begin{align*}
|a|,|b| &\leq 2^{k+2} - 2^j \text{ and} \\
|a|+|b| &\leq 2^{k+2} + 2^{k+1} - 2^j.
\end{align*}
From this we can see that $j \leq k+1$, so 
\begin{align*}
|a|,|b| &\leq 2^{k+2} + 2^{k+1}- 2^{k+1} -2^j &\leq w_{n+1} -2^{j+1} \text{ and}\\
|a|+|b| &\leq 2^{k+3} - 2^{k+1} - 2^j &\leq w_{n+2} - 2^{k+1} -2^j.
\end{align*}
If $j \leq k$, then $|a|+|b| \leq w_{n+2} - 3\cdot 2^j$; if $j=k+1$ and if $a+bi \in 2^{k+1}(\mathbb{Z}[i]^{\times})$, then 
$|a|+|b| = 2^{k+1} \leq w_{n+2} - 3\cdot 2^{k+1}$.  In both of these situations, then $a+bi \in B_{n+1}$.   Our final scenario
is if $a+bi \in 2^{k+1}(1+i)(\mathbb{Z}[i]^{\times}) \subset Oct_n;$ note that $2^{k+1}(1+i) \notin B_{n+1}$.

\end{proof}

We can also use Theorem \ref{B_union} to show the following useful result.

\begin{coro} \label{divides_xy} If $xy \in B_n \setminus 0$, then $x \in B_n$.  
\end{coro}
\begin{proof}
Let $x = a+bi$ and $y = c+di$.  Since $B_n$ is closed under complex conjugation and multiplication by units, we can assume that $a,b \geq 0$ and $c >0$.  We can break our proof into three cases.

\underline{$d=0$}: If $d=0$ and $2^m \parallel (ac,bc)$, then $ac, bc \leq w_n -2^{m+1}$ and $ ac+bc \leq w_{n+1} - 3 \cdot 2^m$.  If $2^k \parallel (a,b)$, implying that $0 \leq k \leq m \leq \frac{n}{2}$, then 
\begin{align*}
|a|,|b| &\leq w_n - 2^{m+1} &\leq w_n - 2^{k+1} \text{ and}\\
|a| + |b|&\leq w_{n+1} - 3 \cdot 2^m &\leq w_{n+1} - 3 \cdot 2^k,
\end{align*} so $x = a+bi \in B_n$.

\underline{$d>0$}: The product $xy = (ac-bd) + (bc+ad)i$, with $bc,ad \geq 0$.  If $2^k \parallel (a,b)$ and $2^m \parallel (ac-bd, bc +ad)$, then 
$0 \leq k \leq m \leq \frac{n}{2}$.  Since $c, d \neq 0$, we have 
\begin{equation*}
|a|,|b| \leq |a|+|b| \leq bc + ad \leq w_n - 2^{m+1} \leq w_n - 2^{k+1} \leq w_{n+1} - 3 \cdot 2^k.
\end{equation*}  We conclude that $x =a+bi \in B_n$.

\underline{$d<0$}: Let us rewrite $y$ as $c - |d|i$.  The product $xy = (ac+b|d|) + (bc - a|d|)i$, with $ac, b|d| \geq 0$.  If $2^k \parallel (a,b)$ and $2^m \parallel (ac + b|d|, bc -a|d|)$, then $0\leq k \leq m \leq \frac{n}{2}$.  Since $c,d \neq 0$, we have 
\begin{equation*}
|a| , |b| \leq |a|+|b| \leq ac + b|d| \leq w_n - 2^{m+1} \leq w_n - 2^{k+1} \leq w_{n+1} - 3 \cdot 2^k,
\end{equation*}
so $x = a+bi \in B_n$. 
\end{proof}

The goal of the rest of the paper is to show that $A_{\mathbb{Z}[i], n} = B_n$.  To show this, we will need several facts about the sets $A_{\mathbb{Z}[i],n}$.

\section{The sets $A_{\mathbb{Z}[i],n}$}\label{A_sets}

It is well-known that $\mathbb{Z}[i]$ is norm-Euclidean (i.e. the algebraic norm $N(a+bi) = a^2 + b^2$ is a Euclidean function on $\mathbb{Z}[i]$), and therefore Motzkin's Lemma (Lemma \ref{Motzkins_Lemma}), implies that $\mathbb{Z}[i] = \bigcup_{n=0}^{\infty} A_{\mathbb{Z}[i],n}$.

We remind ourselves that 
\begin{align*}
A_{\mathbb{Z}[i], 0} &:= \{0, \pm 1, \pm i \}\\
\intertext{and that, for $n \geq 1$,}
A_{\mathbb{Z}[i],n} &:= A_{\mathbb{Z}[i],n-1} \cup \{a+bi \in \mathbb{Z}[i] :A_{\mathbb{Z}[i], n-1} \rightarrow \mathbb{Z}[i]/(a+bi) \}.
\end{align*}  It is clear from this formulation that the $A_{\mathbb{Z}[i], n}$ are closed under multiplication by a unit, as the ideals generated by $a+bi$ and $u(a+bi)$, $u \in \mathbb{Z}[i]^{\times}$, are the same.  It is trickier to see that the sets $A_{\mathbb{Z}[i], n}$ are closed under complex conjugation.

\begin{lemma}  The sets $A_{\mathbb{Z}[i],n}$ are closed under complex conjugation.
\end{lemma}
\begin{proof} We prove this by induction; note that $A_{\mathbb{Z}[i],0} = \{0, \pm 1, \pm i\}$ is closed under complex conjugation.  
Suppose that $A_{\mathbb{Z}[i],n}$ is closed under complex conjugation, that $a+bi \in A_{\mathbb{Z}[i], n+1}$, and that $[x] \in \mathbb{Z}[i] / (\overline{a+bi})$.  We know that there exists some $q$ in $\mathbb{Z}[i]$ and some $r \in A_{\mathbb{Z}[i], n}$ such that 
\begin{align*}
\overline{x} &= q (a+bi) + r, \text{ or }
x = \overline{q} (\overline{a+bi} ) + \overline{r}.
\end{align*}
Our induction hypothesis forces $\overline{r}$ to be an element of $A_{\mathbb{Z}[i], n}$, so $\overline{a+bi} \in A_{\mathbb{Z}[i], n+1}$.
\end{proof}

\begin{coro}\label{you_get_the_whole_set} An element $a+bi \in A_{\mathbb{Z}[i],n}$ if and only if $\{ \pm a \pm bi \}, \{ \pm b \pm ai\} \subset A_{\mathbb{Z}[i],n}$.
\end{coro}

We can use the properties above the prove to $A_{\mathbb{Z}[i], n} = B_n$, which gives the minimal Euclidean function on $\mathbb{Z}[i]$.
In order to prove the equality, we start with the following lemmas.

\begin{lemma} \label{containment}If $A_{\mathbb{Z}[i], n }= B_n$, then $A_{\mathbb{Z}[i], n+1} \subset B_{n+1}$.
\end{lemma}
\begin{proof}  Given $a+bi \in A_{\mathbb{Z}[i], n+1}$, there exists some $q \in \mathbb{Z}[i]$ and $r \in A_{\mathbb{Z}[i], n}$ such that 
$(1+i)^{n+1} = q(a+bi) +r$.  Rearranging terms reveals that 
\begin{equation*}
q(a+bi) = (1+i)^{n+1} - r \in B_{n+1} \setminus 0,
\end{equation*}
so Corollary \ref{divides_xy} tells us that $a+bi \in B_{n+1}$.
\end{proof}

\begin{lemma}\label{multiply_by_1+i}  If $A_{\mathbb{Z}[i], j} = B_j$ for all $j \leq n$, then $(1+i)B_n \subset A_{\mathbb{Z}[i], n+1}$.
\end{lemma}
\begin{proof} We start with $a \in B_n$.  Given $x \in \mathbb{Z}[i]$, there exists some $q\in \mathbb{Z}[i]$ and $r \in A_{\mathbb{Z}[i], 0}$ such that $x = (1+i) q +r$.  Since $a \in A_{\mathbb{Z}[i], n}$, there exists $q' \in \mathbb{Z}[i]$ and $r' \in A_{\mathbb{Z}[i], n-1} = B_{n-1}$ such that $q = q' a + r'$, so
\begin{align*}
(1+i)q + r &= (1+i)q' a + (1+i)r' +r,\\
\intertext{and thus}
x &= q' (1+i)a + ((1+i)r'+r).
\end{align*}
The element $(1+i)r' + w \in B_n = A_{\mathbb{Z}[i], n}$, so $a(1+i) \in A_{\mathbb{Z}[i], n+1}$.
\end{proof}
 %
 
All that remains for us to show is that if $a+bi \in B_{n+1} \setminus B_n$, $(1+i) \nmid a+bi$, then $B_n \twoheadrightarrow \mathbb{Z}[i]/(a+bi)$.  If $S \subset \mathbb{Z}[i]$ is a set that contains a representative of every coset in $\mathbb{Z}/(a+bi)$ and if $S \subset 
\bigcup_{x \in \mathbb{Z}[i]} (B_n + x(a+bi))$, then $B_n \twoheadrightarrow \mathbb{Z}[i]/(a+bi)$.  The next section will explore potential such sets $S$.

\section{Sets of Representatives of Cosets of $a+bi$}

Since we are working in the Gaussian integers, assume that all the sets mentioned below in this section ( the sets $S, T, U,$ and $\mathscr{S}$)are all subsets of $\mathbb{Z}[i]$.  

\begin{lemma}\label{a_square}  If $a >b \geq 0$; if $k_0, k_1 \in \mathbb{Z}$; if 
\begin{equation*}
S = \{ x+yi: k_0 \leq x < k_0 +a, k_1 \leq y < k_1 + a\};
\end{equation*}
and if $\alpha + \beta i, c + di$ are distinct elements of $S$, then $\alpha + \beta i \not \equiv c+di \pmod{a + bi}$.
\end{lemma}
\begin{proof} Suppose, leading to a contradiction, that $\alpha + \beta i \equiv c+di \pmod{a+bi}$.  In other words, suppose that there exists some $y \in \mathbb{Z}[i]$ such that $(\alpha - c) + (\beta -d) i = y (a+bi)$.  
Note that 
\begin{equation*}
Nm(y) Nm(a+bi) = (\alpha -c)^2 + (\beta -d)^2 \leq 2(a-1)^2 < 2(a^2 + b^2)=2 Nm(a+bi).
\end{equation*}
The norm of $y$ is a positive integer, so it must equal one, implying that $y \in \mathbb{Z}[i]^{\times}$.  We conclude that 
\begin{equation*}
(\alpha - c) + (\beta -d)i \in \{ \pm (a+bi), \pm (b-ai)\},
\end{equation*}
but as $|\alpha -c|, |\beta -d| \leq a-1$, this is a contradiction.
\end{proof}

\begin{lemma} \label{two_squares} If $a > b \geq 0$, if $T=\{ x+iy: 0 \leq x <b, -b \leq y <0\},$ if 
$S = \{ x+yi: 0 \leq x,y < a\}$, 
and if $\alpha + \beta i, c+di$ are distinct elements of $S \cup T$, then $\alpha + \beta i \not \equiv c+di \pmod{a +bi}$.
\end{lemma}
\begin{proof} We know from Lemma \ref{a_square} that two distinct elements of $S$ (respectively, $T$) are not equivalent modulo $a+bi$.  It remains to show that if $\alpha + \beta i \in T$ and $c+di \in S$, then $\alpha + \beta i \not \equiv c+di \pmod{a+bi}$.  

Suppose, leading to a contradiction, that $\alpha + \beta i \not \equiv c+di \pmod{a+bi}$, i.e. that there exists some $z \in \mathbb{Z}[i]$ such that $(\alpha + \beta i) - (c+di) = z(a+bi)$.  Note that 
\begin{align*}
Nm(z)Nm(a+bi) &= (\alpha - c)^2 + (\beta -d)^2 \\
& \leq (a-1)^2 + (a+b-1)^2\\
& < a^2  + b^2 + 3 a^2 \\
& < 4 (a^2 + b^2) = 4 Nm(a+bi).
\end{align*}
This implies that $1 \leq Nm(z) < 4$, so $Nm(z) = 1$ or $2$, as there is no Gaussian integer with norm $3$.  We can list all the Gaussian integers with norm 1 or 2, so we can see that $z \in \{ \pm 1, \pm i, \pm 1 \pm i \}$ 
and thus
\begin{equation*}
z(a+bi) \in C = \{ \pm (a+bi), \pm (b-ai), \pm (a-b + (a+b)i), \pm (a + b + (b-a)i) \}.
\end{equation*}
It is easy to check that $(\alpha + \beta i + C) \cap S = \emptyset$.
\end{proof}
\begin{coro} \label{full_set} If $a > b \geq 0$, if $T=\{ x+iy: 0 \leq x <b, -b \leq y <0\},$ and if 
$S = \{ x+yi: 0 \leq x,y < a\}$, then $S\cup T$ contains exactly one representative for each coset of $a+bi$.
\end{coro}
\begin{proof} Lemma \ref{two_squares} already showed that each element in $S \cup T$ belongs to a distinct coset of $a+bi$.  There are $a^2 + b^2 = Nm(a+bi)$ elements in $S\cup T$, so it contains a representative from each coset.
\end{proof}

\begin{lemma}\label{triangle} Let $a > b \geq 0$, $2 \nmid (a,b)$, and let 
\begin{equation*}
\mathscr{S} = \{ x+yi: 0 \leq x <a, 0 \leq y < a-x\}.
\end{equation*}  If 
$\mathscr{S} \subset \bigcup_{z \in \Z[i]} (B_n + (a+bi)z)$, then $B_n \twoheadrightarrow \mathbb{Z}[i]/(a+bi)$.

\begin{proof} For ease of exposition, let $U = \bigcup_{z \in \Z[i]} (B_n + (a+bi)z)$ and note that $U$ is closed under multiplication by units.  If 
\begin{align*}
S &= \{x+yi: 0 \leq x,y < a\} \subset U,\\
\intertext{then}
T&= \{ x+yi: 0 \leq x < b, -b \leq y <0\} \subset U,
\end{align*}
as $T \subset -iS \subset -iU \subset U$.  
Corollary \ref{full_set} states that $S \cup T$ has a representative of every class in $\mathbb{Z}[i]/(a+bi)$, so if $S \cup T \subset U$, then 
$B_n \twoheadrightarrow \mathbb{Z}[i]/(a+bi)$.  To prove our lemma, then, it suffices to show that if $\mathscr{S} \subset U$, then $S \subset U$.

Note that if $\mathscr{S} \subset U$, then $(-\mathscr{S} + a + bi), (i \mathscr{S} + a + bi) \subset U$.  The union 
$\mathscr{S} \cup ( -\mathscr{S} + a + bi) \cup (i \mathscr{S} + a + bi)$ contains the set 
$
\{x+yi : 0 \leq x < a, 0 \leq y < x+b\}.
$
The set $U$ also contains $i(-\mathscr{S} + a +bi)$, which itself contains the set \\
$\{x+yi: 0 \leq x < a-b, x+b < y \leq a\}$, so $U$ contains 
\begin{equation*} \label{p_1}
\{x+yi: 0 \leq x,y < a, y \neq x+b \} \text{ or } S \setminus \{x+yi: y = x+b \}.
\end{equation*}

The set $-\mathscr{S} + (1+i)(a+bi)$ contains $\{x+yi: 0 \leq x \leq a-b, a-x < y <a\}$, so if we take the union with $\mathscr{S}$, we see that 
$U$ contains 
\begin{equation*}\label{p-2}
\{x+yi: 0 \leq x,y < a, y \neq a-x \} \text{ or } S \setminus \{x+yi: y = a-x \}.
\end{equation*}

These two expressions show that 
$(S \setminus \{ x+yi: y = a-x = x+b\}) \subset U$.  If $a-x = x+b$, then $2x = a-b$, and $x$ cannot be integral, as $2 \nmid (a,b)$.  We conclude that $S \subset U$.

\end{proof}

\end{lemma}

 \section{Our Main Result}

We can now look at the question of whether $B_n \twoheadrightarrow \mathbb{Z}[i]/(a+bi)$ geometrically, by looking at sets that cover our set $\mathscr{S}$.  The next lemma demonstrates why we will be interested in translates of $Oct_n$. 
\begin{lemma}\label{oct_translate} If $2 \nmid(a,b)$ , if $u \in \mathbb{Z}[i]^{\times}$, if $x+yi \in (Oct_n + u(a+bi))$, and if $2|(x,y)$, then $x+yi \in (B_n + u(a+bi))$.
\end{lemma}
\begin{proof}  If $x+yi \in (Oct_n + u(a+bi))$, then there exists some $c +di \in Oct_n$ such that $(x+yi) - u(a+bi) = c+di$.  As $2\nmid (a,b)$ and $2|(x,y)$, we know that $2 \nmid (c,d)$ and thus $c+di \in B_n$ by Corollary \ref{oct}.
\end{proof}
 

The aim of this paper is to prove that $A_{\mathbb{Z}[i], n} = B_n$.  As we saw in section \ref{A_sets}, we must still show that if $(a+bi) \in B_{n+1}\setminus B_n$, $(1+i)\nmid (a+bi)$, then $B_n \twoheadrightarrow \mathbb{Z}[i]/(a+bi)$.  We will prove that geometrically, using Lemmas \ref{triangle} and \ref{oct_translate}.  We will break our problem up into cases, and prove them in the following results.  

\begin{lemma} \label{small} If $a+bi \in B_{n+1} \setminus B_n$, if $(1+i) \nmid a+bi$, and if $|a|,|b| \leq w_n -2$, then $B_n \twoheadrightarrow \mathbb{Z}[i]/(a+bi)$.
\end{lemma}
\begin{proof}  For ease of notation, let $U = \bigcup_{z \in \mathbb{Z}[i]} (B_n + z(a+bi))$.  Since $(1+i)\nmid (a+bi)$ and since $B_{n+1},B_n$ are both closed under complex conjugation and multiplication by units, we can assume without loss of generality that 
$w_{n} -2 \geq a > b \geq 0$, $2 \nmid (a,b)$.  

Let us define $\mathscr{S} = \{ x+yi: 0 \leq x, y; x+y <a\}$, and let us note that 
\begin{equation*}
\mathscr{S} \subset \{ x+yi : 0 \leq x,y; x+y < w_n -2\} \subset Oct_n.
\end{equation*} 
Corollary \ref{oct} then tells us that if $x+yi \in \mathscr{S}, 2 \nmid (x,y)$, then $x+yi \in B_n \subset U$.


As $b -(w_n-2) < a - (w_n -2) \leq 0$, the lower left edge of $Oct_n + a+bi$ intersects $\mathscr{S}$ at $y = (3+a+b-w_{n+1}) -x$.  If $c+di \in 
\mathscr{S}$, if $2 |(c,d)$, and if $c+d \geq 3 + a_b - w_{n+1}$, then $c+di \in (Oct_n + a+bi)$ and $c+di \in (B_n + a+bi) \subset U$ by Lemma 
\ref{oct_translate}.  Now suppose that $c+di \in \mathscr{S}$, with $2 |(c,d)$ and $c+d < (3 + a+b-w_{n+1}).$  We know that 
\begin{equation*}
3 + a+b -w_{n+1} \leq 3 + w_{n+2} - 3 - w_{n+1} \leq w_{n+2} - w_{n+1} = 2(w_n - w_{n-1}) \leq w_{n-1},
\end{equation*}
so by Corollary \ref{oct}, $c+di \in B_n \subset U$.  As $\mathscr{S} \subset U$, we apply Lemma \ref{triangle} and conclude that 
$B_n \twoheadrightarrow \mathbb{Z}[i]/(a+bi)$.
\end{proof}

We can now prove our main result. 
\begin{theorem} \label{main_result} For $n \geq 0$, $A_{\mathbb{Z}[i],n} = B_n$.
\end{theorem} 

\begin{proof}  We will prove the theorem using induction.  We see upon examination that $A_{\mathbb{Z}[i],n} = B_n$ for $n =0,1$.  Suppose that $A_{\mathbb{Z}[i],j} = B_j$ for all $j \leq n$ and that $n \geq 2$.  Lemma \ref{containment} states that $A_{\mathbb{Z}[i],n+1} \subset B_{n+1}$, so to prove the equality, we need to show that $B_{n+1} \subset A_{\mathbb{Z}[i], n+1}$.  

It is clear that $B_n \subset A_{\mathbb{Z}[i], n+1}$ and Lemma \ref{multiply_by_1+i} establishes that $(1+i)B_n \subset A_{\mathbb{Z}[i], n+1}$.  It therefore remains to show that $B_{n+1}\setminus (B_n \cup (1+i)B_n) \subset A_{\mathbb{Z}[i], n+1}$.  We can see, however, that $B_{n+1} \setminus (B_n \cup (1+i)B_n)$ equals the set of $a+bi \in B_{n+1}\setminus B_n$ such that $(1+i) \nmid (a+bi)$.  Lemma \ref{small} already proved that if $a+bi \in B_{n+1}\setminus B_n$ such that  $(1+i) \nmid (a+bi)$ and $|a|, |b| \leq w_n -2$, then $a+bi \in A_{\mathbb{Z}[i], n+1}$.  It remains to show that if $a+bi\in B_{n+1} \setminus B_n$, $(1+i) \nmid (a+bi)$ and $\max (|a|, |b|) > w_n -2$, then $B_n \twoheadrightarrow \mathbb{Z}[i]/(a+bi)$.

 We will prove this using Lemma \ref{triangle}, by showing that 
\begin{equation*} 
\mathscr{S} = \{x+yi: 0 \leq x,y; x+y < a \} \subset U = \bigcup_{z\in \mathbb{Z}[i]} (B_n + z(a+bi)).
\end{equation*} Corollary \ref{oct} states that if $x+yi \in \mathscr{S} \cap Oct_n$, $2 \nmid (x,y)$, then $x+yi \in B_n \subset U$. As $a -1 \leq w_{n+1} - 3$, Corollary \ref{oct} implies that $\{x+yi \in \mathscr{S}: 2\nmid (x,y), \max (x,y) \leq w_n -2\} \subset U$.  We will now break our proof into two cases.

\underline{$\mathbf{b\leq w_{n+1}-w_n:}$ }  We will first look at the pairs $x+yi \in \mathscr{S}$ such that $2 \nmid (x,y)$.  We already know that if such $x+i \in Oct_n$, then they are in $B_n$, so we will study the pairs in $\mathscr{S}\setminus Oct_n$.  Suppose that $x+yi \in \mathscr{S}$, $2 \nmid (x,y)$, and that $y> w_n -2$, so that 
\begin{equation*}
x < a-y < w_{n+1} - 2 - w_n + 2 = w_{n+1} - w_n.
\end{equation*}
If $2^k \parallel (a-y, b+x)$, then 
$x < a-y \leq w_{n+1} - w_n - 2^k \leq w_n -2^{k+1}$ and 
\begin{align*}
b+x &< w_{n+1} - w_n + w_{n+1} - w_n - 2^k = w_n - 2^k, \text{ so}\\
b+x &\leq w_n - 2^{k+1}.
\end{align*}
From this we infer that that 
\begin{align*}
a-y+b+x &\leq w_{n+1} -w_n - 2^k + w_n - 2^{k+1}=w_{n+1} - 3\cdot 2^k, \text{ so}\\
-b+ai - (x+yi) &= -(b+x) + (a-y)i\in B_n, 
\end{align*}
and $x+yi \in U$.  The set $U$ is closed under multiplication by $-i$, so $y -xi \in U$ and $y \in U$.  

Now suppose that $x+yi \in \mathscr{S}$, $2 \nmid (x,y)$, and that $x > w_n -2$.  We can assume that $y>0$ as we just showed that $x \in U$ if 
$2\nmid x$, $x > w_n -2$.  If $2^k \parallel (a-x, b-y)$, then 
\begin{align*}
y &< a-x  \leq w_{n+1} - w_n - 2^k \text{ and}\\
|b-y| &\leq w_{n+1} - w_n - 2^k, \text{ so}\\
a-x + |b-y| &\leq 2(w_{n+1} -w_n) - 2^{k+1}\\
 &\leq w_n - 2^{k+1}\\
&\leq w_{n+1} - 3\cdot 2^k.
\end{align*}  
This implies that $(a-x) + (b-y)i \in B_n$ and $x+yi \in U$, so we conclude that $\{x+yi \in \mathscr{S}, 2 \nmid(x,y) \} \subset U$.

Let us turn our attention to the $x+yi \in \mathscr{S}$, $2^k \parallel (x,y)$ with $k \geq 1$.  Before proceeding, note that 
$a-(w_n -2) \leq w_{n+1} - w_n \leq w_{n-2},$ so $a-(w_n -2)$ and $(a-(w_n -2))i \in B_n$.  Also note that $a - (w_n -2) \leq (w_n -2) - b+1$.  

The set $Oct_n + (a+bi)$ contains the set 
\begin{equation*}
\{x+yi: a-(w_n-2) \leq x < a, 0 \leq y  \leq \min\{a-1-x, w_{n+1} + b -a -3 +x\} \},
\end{equation*}
with the potential exception of the element $a-(w_n -2)$, so
\begin{equation*}
\{x+yi: 2 |(x,y), a-(w_n-2) \leq x < a, 0 \leq y < \min\{a-x, w_{n+1} + b -a -2 +x\} \} 
\end{equation*}
is contained in $U$. 
Doing the same sort of analysis on $Oct_n +i(a+bi)$, we see that 
\begin{equation*}
\{x+yi: 2|(x,y), \max\{a-(w_n-2), a + b + 3 -w_{n+1} + x \} \leq y < a-x, 0 \leq x \leq  -b + w_n -2 \}
\end{equation*}  
is also contained in $U$.  Upon examining these sets, we see that 
\begin{equation*}
(w_{n+1} +b -a - 2 + x) +1 \geq a + b + 3 -w_{n+1} +x
\end{equation*}
and that if $y = w_{n+1} + b -a-2+x$, then $2\nmid (x,y)$.  The union of our two sets contained in $U$ thus covers 
$\{x+yi \in \mathscr{S} : 2 |(x,y)\}$ with the exception of the set 
\begin{equation*}
\mathscr{S}^{\star} = \{ x +yi \in \mathscr{S}: 2|(x,y), x < a-(w_n-2), y < \max \{ a-(w_n -2), a+b + 3 -w_{n-1} + x\} \}.
\end{equation*}
Suppose that $x+yi \in \mathscr{S}^{\star}$ with $2^k \parallel (x,y)$, so that $x \leq w_{n+1} - w_n -2^k$.  If $y < a- (w_n-2)$,
then $x+y \leq 2(w_{n+1} - w_n -2^k) \leq w_n - 2^{k+1} \leq w_{n+1} - 3\cdot 2^k,$ and $x+yi \in B_n \subset U$.  
If 
\begin{align*} 
y &< a+b + 3 - w_{n+1} + x, \text{ then}\\
y &< w_{n+2} -3 + 3 - w_{n+1} + w_{n+1} - w_n - 2^k \text{ and}\\
y &\leq w_n - 2^{k+1}.
\end{align*}
From this, we see that 
$x+y \leq w_{n+1} -w_n - 2^k + w_n - 2^{k+1} = w_{n+1} - 3 \cdot 2^k$, and therefore $x+yi \in B_n \subset U$.  We have now shown that 
$\{x+yi \in \mathscr{S} : 2 |(x,y) \} \subset U$, and we conclude that $\mathscr{S} \subset U$.

\underline{ $\mathbf{b > w_{n+1} - w_n}:$}  We again start with the set $\mathscr{S}\setminus Oct_n$.  Suppose that $x+yi \in \mathscr{S}$, that $2^k \parallel (a-x,b-y)$, and that $x > w_n -2$, so that 
$y < (a-x) \leq w_{n+1} - w_n -2^k$.  If 
\begin{align*} b-y &\leq w_n - 2^{k+1}, \text{ then}\\
a-x+b-y &\leq w_{n+1} - 3 \cdot 2^k
\end{align*}
and $x+ yi \in B_n \subset U$.  
If $b-y > w_n - 2^{k+1}$ so that $b -y \geq w_n - 2^k$ and 
\begin{align*}
|b-a-y| = a-b+y &\leq w_{n+1} - 2 + 2^k - w_n \\
&\leq w_{n+1} - w_n -2 + w_{n+1} - w_n \\ 
&\leq w_n - 2.
\end{align*} 
We already know that $a+b-x \leq w_{n+1} -3+ 1 - w_n = w_n -2$, so 
\begin{align*}
a+b -x + a-b +y &\leq w_{n+2} - 3 + w_{n+1} - w_n - 2^k + 2^k - w_n \\
&\leq w_{n+1} - 3.
\end{align*}
As $x$ and $y$ have different parities, and as $a$ and $b$ have different parities, $2 \nmid (a+b-x, b-a-y)$, implying that $(1+i)(a+bi) - (x+yi) \in B_n$ and $x+yi \in U$.

Now suppose that $x+yi \in \mathscr{S}$, that $2^k \parallel (a-x,-y)$, that $2 \nmid (x,y)$, and that $y > w_n -2$, so that 
$x< a-y\leq w_{n+1} - w_n -2^k$.  If $b+x \leq w_n - 2^{k+1}$, then $a-y+b+x \leq w_{n+1} - 3 \cdot 2^k$, so $-(b+x) + (a-y)i \in B_n$ and $x+yi \in U$.  If $b+x > w_n - 2^{k+1}$, then $b+x \geq w_n -2^k$.  Note that 
\begin{equation*}
b+x < a+b-y \leq w_{n+2} -3 + 1 - w_n = w_n -2, \text{ so}
\end{equation*} 
\begin{align*}
|a-b-x| = a-(b+x) &\leq w_{n+1} - 2 + 2^k - w_n\\
&\leq 2(w_{n+1} - w_n) - 2\\
&\leq w_n -2.
\end{align*}  
We also note that 
\begin{align*}
a-(b+x)  + (a+b -y) &\leq a - y + a + b - w_n -2^k \\
&\leq w_{n+1} - w_n - 2^k + w_{n+2} - 3 -w_n + 2^k\\
&= w_{n+1} -3
\end{align*}
and $2 \nmid (a-b-x, a+b-y) $, so $(a+bi)(1+i) - (x+yi) \in B_n$.  We conclude that  as $x+yi \in U$, the set $\{x+yi \in \mathscr{S}, 2\nmid (x,y) \} \subset U$.   

We again turn our attention to the $x+yi \in \mathscr{S}$ such that $2^k \parallel (x,y), k \geq 1$.  The octogon $Oct_n + a +bi$ intersects $\mathscr{S}$ at the lines $x= (a-(w_n -2))$ and $y = (a+b+3-w_{n+1}) -x$, so 
\begin{equation*}
\{ x+yi: a-(w_n-2) \leq x < a, a+b+3 - w_{n+1} - x \leq y < a-x\} \subset Oct_n + a +bi.
\end{equation*}
Note that our set contains 
$\{ x+yi \in \mathscr{S}, x \geq w_{n+2} - w_{n+1} \}$.  If $0 < x < w_{n+2} - w_{n+1}$, $0\leq y < a+b+3- w_{n+1} -x$, then 
\begin{align*}
0 < x &\leq w_{n+2} - w_{n+1} -2^k \text{ and}\\
y &< w_{n+2} - w_{n+1} -2^k, \text{ so}\\
y &\leq w_{n+2} - w_{n+1} - 2^{k+1} \leq w_n - 2^{k+1}.
\end{align*}
We can use these two inequalities to see that 
\begin{align*}
x+y &\leq w_{n+2} - w_{n+1} - 2^k + w_{n+2} - w_{n+1} - 2^{k+1}\\
&\leq w_{n+1} - 3 \cdot 2^k,
\end{align*}
so $x+yi \in B_n$, and 
\begin{align*}
\{x+yi \in \mathscr{S}: 2 | (x,y), x \geq a - (w_{n-2} -2) \} \subset U.
\end{align*}

The octogon $Oct_n -b +ai$ intersects $\mathscr{S}$ at the lines $x= w_n -2-b$ and $y = (a+b+3 - w_{n+1}) + x$, so 
\begin{equation*}
\{x+yi: 0 \leq x \leq w_n - 2-b, a+b+3- w_{n+1} + x \leq y < a-x\} \subset (Oct_n -b+ai).
\end{equation*}
Recall that $(w_n - 2 -b) \geq (a-(w_n -2)) -1$.  
If 
\begin{align*}
&0 \leq x < a-(w_n -2) &\text{and } &0 \leq y < a+b + 3 - w_{n+1} +x, \text{ then}\\ 
&0 \leq x \leq w_{n+1} - w_n - 2^k &\text{and } &0 \leq y < w_{n+2} - w_{n+1} + w_{n+1} - w_n - 2^k,
\end{align*} implying that 
$0\leq y \leq w_n - 2^{k+1}$.  From this, we observe that 
$x+y \leq w_{n+1} - 3 \cdot 2^k$, so $x+yi \in B_n \subset U$.  In summary, we have now shown that $\{x+yi \in \mathscr{S}, 2|(x,y) \} \subset U$, so $\mathscr{S} \subset U$. QED.

\end{proof}

\section*{Acknowledgements}
I would like to thank my very patient spouse, Loren LaLonde, who has listened to me talk about this problem for the last twelve years, and who has ensured that my LaTeX always compiled.

I greatly appreciate the help from Michael Bridgeland, who not only read Martin Fuch's thesis for me (I don't speak German), but who also asked a question that led to Corollary \ref{oct}.  

Lastly, I would like to thank both H.W. Lenstra, Jr.  and Franz Lemmermeyer for their help with this paper's background research and literature review.  I highly recommend Lemmermeyer's survey, ``The Euclidean Algorithm in Algebraic Number Fields,'' to anyone interested in the subject.


\begin{thebibliography}{99}

\bibitem{Fuchs} M. Fuchs,
\newblock{``Der Minimale Euklidische Algorithmus im Ring der ganzen Gausschen Zahlen,''}
\newblock{Munich}, 2003.
\newblock{http://www.mafu.ws/papers/studienarbeit.pdf}

\bibitem{Lemmermeyer} F. Lemmermeyer,
\newblock{``The Euclidean Algorithm in Algebraic Number Fields,''} 2004.
\newblock{http://www.rzuser.uni-heidelberg.de/~hb3/publ/survey.pdf}

\bibitem{Lenstra} H.W. Lenstra, Jr.,
\newblock{``Lectures on Euclidean Rings,''}
\newblock{Bielefield}, 1974.

\bibitem{Motzkin} T. Motzkin, 
\newblock{``The Euclidean Algorithm,''}
\newblock{ \sl Bull. Am. Math. Soc}, 55 (1949), 1142-1146.

\bibitem{Samuel} P. Samuel,
\newblock{``About Euclidean Rings,''}
\newblock{ \sl J. Algebra}, 19 (1971), 282-301.  

 
 
 
 
 

\end{thebibliography}
\end{document}